\documentclass{proc-l}%

\usepackage{amscd,bbm,amsmath,amsthm,graphicx,amssymb,mathrsfs,latexsym,dsfont,amsfonts,paralist,epsfig,graphics,exscale}
\usepackage[colorlinks=true]{hyperref}
\usepackage{color,soul}%%%\hl{ highlighting text } % use hl{ } to highlight the sentence
\newcommand{\bF}{\boldsymbol{F}}
\newcommand{\bS}{\boldsymbol{S}}

\newcommand{\balpha}{\boldsymbol{\alpha}}

\newcommand{\bcdot}{\boldsymbol{\cdot}}

\newtheorem{theorem}{Theorem}[section]
\newtheorem{cor}[theorem]{Corollary}
\newtheorem*{Mainthm}{Main theorem}

\newtheorem*{Kingman}{Kingman's Subadditive Ergodic Theorem}

\newcommand{\EQ}{\begin{equation}\begin{array}{lllllllll}}
\newcommand{\EE}{\end{array}\end{equation}}
\newcommand{\MT}{\left[ \begin{array}{ccccccccc}}
\newcommand{\EM}{\end{array}\right]}

\newcommand{\eq}{\begin{equation}\begin{array}{lclllllllllllllll}}
\newcommand{\ee}{\end{array}\end{equation}}
\newcommand{\bmt}{\left[ \begin{array}{ccccccccc}}
\newcommand{\emt}{\end{array}\right]}
\newcommand{\bea}{\begin{eqnarray}}
\newcommand{\eea}{\end{eqnarray}}
\newcommand{\bean}{\begin{eqnarray*}}
\newcommand{\eean}{\end{eqnarray*}}

\newtheorem{lem}[theorem]{Lemma}

\theoremstyle{definition}
\newtheorem{notes}{\bf Note}

\numberwithin{equation}{section}
\newcommand{\norm}[1]{\left\Vert#1\right\Vert}

\begin{document}

\title[Extremal ergodic measures and the finiteness property]{Extremal ergodic measures and the finiteness property of matrix semigroups}%%%%
\thanks{Project was supported partly by National Natural Science Foundation of China (Nos.~11071112 and 11071263) and in part by NSF DMS-0605181, 1021203 of the United States.}%%%%%%

\author{Xiongping Dai}
\address{Department of Mathematics, Nanjing University, Nanjing 210093, People's Republic of China}
\email{xpdai@nju.edu.cn}

\author{Yu Huang}
\address{Department of Mathematics, Zhongshan (Sun Yat-Sen) University, Guangzhou 510275, People's Republic of China}
\email{stshyu@mail.sysu.edu.cn}

\author{MingQing Xiao}
\address{Department of Mathematics, Southern Illinois University, Carbondale, IL 62901-4408, USA}
\email{mxiao@math.siu.edu}
%%%%%%%%%%%%%%%%%%%%%%%%%%%%%%%%%%%%%%

\date{June 10, 2011 and, in revised form, June 27, 2011.}

%%%%%%%%%%%%%%%%%%%%%%%%%%%%%%%%%%%%%%
\begin{abstract}
Let $\bS=\{S_1,\ldots,S_K\}$ be a finite set of complex $d\times d$ matrices and
$\varSigma_{\!K}^+$ the compact space of all one-sided infinite sequences $i_{\bcdot}\colon\mathbb{N}\rightarrow\{1,\dotsc,K\}$.
An ergodic probability $\mu_*$ of the Markov shift $\theta\colon\varSigma_{\!K}^+\rightarrow\varSigma_{\!K}^+;\ i_{\bcdot}\mapsto i_{\bcdot+1}$,
is called ``extremal" for $\bS$, if
${\rho}(\bS)=\lim_{n\to\infty}\sqrt[n]{\norm{S_{i_1}\cdots S_{i_n}}}$ holds for $\mu_*$-a.e. $i_{\bcdot}\in\varSigma_{\!K}^+$, where $\rho(\bS)$ denotes the generalized/joint spectral radius of $\bS$. Using extremal norm and Kingman subadditive ergodic theorem, it is shown that $\bS$ has the spectral finiteness property (i.e. $\rho(\bS)=\sqrt[n]{\rho(S_{i_1}\cdots S_{i_n})}$ for some finite-length word $(i_1,\ldots,i_n)$) if and only if for some extremal measure $\mu_*$ of $\bS$, it has at least one periodic density point $i_{\bcdot}\in\varSigma_{\!K}^+$.
\end{abstract}
%%%%%%%%%%%%%%%%%%%%%%%%%%%

\maketitle
%%%%%%%%%%%%%%%%%%%%%%%%%%%%%%%%%%%%%%%%%%%%%%%%%%%%%%%%%%%%%%%%
%%%%%%%%%%%%%%%%%%%%%%%%%%%%%%%%%%%%%%%%%%%%%%%%%%%%%%%%%%%%%%%%
\section{Introduction}\label{sec1}%%%
We consider an arbitrary finite set of complex, $d\times d$, matrices $\bS=\{S_1,\dotsc,S_K\}$, where $d,K$ both are integers with $2\le d,K<+\infty$. Let
$\varSigma_{\!K}^+$ be the compact topological space of all the one-sided infinite sequences $i_{\bcdot}\colon\mathbb{N}\rightarrow\{1,\dotsc,K\}$ with the product topology,
where $\mathbb{N}=\{1,2,\dotsc\}$.
Recall that the \emph{generalized spectral radius} of $\bS$, first introduced by Daubechies and Lagarias in \cite{DL}, is defined as
\bean
{\rho}(\bS)=\limsup_{n\to+\infty}\max_{i_{\bcdot}\in\varSigma_{\!K}^+}\sqrt[n]{\rho(S_{i_1}\dotsm S_{i_n})}
\quad\left(~=\sup_{n\ge1}\max_{i_{\bcdot}\in\varSigma_{\!K}^+}\sqrt[n]{\rho(S_{i_1}\dotsm S_{i_n})}\right).
\eean
Here $\rho(A)$ stands for the usual spectral radius for an arbitrary matrix $A\in\mathbb{C}^{d\times d}$.
Another critical characterization of all infinite
products of the matrices of $\bS$ is the so-called \emph{joint spectral radius} of $\bS$, appeared initially in Rota and Strang~\cite{RS60},
which is given as
\bean
\hat{\rho}(\bS)=\limsup_{n\to+\infty}\max_{i_{\bcdot}\in\varSigma_{\!K}^+}\sqrt[n]{\norm{S_{i_1}\dotsm S_{i_n}}}\quad
\left(~=\inf_{n\ge1}\max_{i_{\bcdot}\in\varSigma_{\!K}^+}\sqrt[n]{\norm{S_{i_1}\dotsm S_{i_n}}}\right),
\eean
where $\norm{\cdot}$ can be any matrix norm satisfying submultiplicativity (also called ring) property, i.e.,
$\norm{AB}\le \norm{A}\cdot\norm{B}$ for all $A,B\in\mathbb{C}^{d\times d}$.
According to the Berger-Wang formula \cite{BW}, it holds that
\bean
{\rho}(\bS)=\hat{\rho}(\bS).
\eean
The generalized/joint spectral radius has been the subject of substantial recent research interest in many pure and applied mathematical branches, such as matrix analysis, control theory, wavelets and so on. If one can find some finite-length word, say $(i_1,\dotsc,i_n)$, such that
${\rho}(\bS)=\sqrt[n]{\rho(S_{i_1}\dotsm S_{i_n})}$,
then $\bS$ is said to have \emph{the spectral finiteness property}. It has been known that the finiteness property does not hold in general. But what key condition implies the finiteness property remains unclear in the theory of joint spectral radius. A brief survey for some recent progresses regarding this can be found in Dai and Kozyakin~\cite[Section~1.2]{DK11}.

Although the finiteness property failed to exist, the idea is
still to be attractive and important due to algorithms for the
computation of the joint spectral radius must be implemented in
finite arithmetic. Thus criteria for determining if a given matrix family satisfies the finiteness property is critical for us to develop a decidable algorithm for the joint/generalized spectral radius.

In this note, we will study the spectral finiteness property of $\bS$ via ergodic theory. In \cite{DHX:ERA11}, Dai \textit{et al.} showed that there exists an ergodic measure $\mu_*$ of the one-sided Markov shift on $\varSigma_{\!K}^+$ such that
\bean
{\rho}(\bS)=\lim_{n\to+\infty}\sqrt[n]{\norm{S_{i_1}\dotsm S_{i_n}}}\quad \textrm{for }\mu_*\textrm{-a.e. }i_{\bcdot}\in\varSigma_{\!K}^+.
\eean
Any such measure is called \emph{extremal} for $\bS$. Our aim of this note is to show that $\bS$ has the finiteness property if and only if there is an extremal measure $\mu_*$ that has a periodic density point; as shown by our main theorem given in Section~\ref{sec4}.

%%%%%%%%%%%%%%%%%%%%%%%%%%%%%%%%%%%%%%%%%%%%%%%%%%%%
%%%%%%%%%%%%%%%%%%%%%%%%%%%%%%%%%%%%%%%%%%%%%%%%%%%%
\section{Preliminaries}\label{sec2}%
In this section, we will introduce some preliminary notations and lemmas needed for the Main theorem in Section~\ref{sec4}.
%%%%%%%%%%%%%%%%%%%%%%%%%%%%%%%
\subsection{Ergodic measures}\label{sec2.1}%%
Let $K\ge2$ be an arbitrary integer. We equip the finite set $\{1,\ldots,K\}$ with the usual discrete topology. Let $\varSigma_{\!K}^+$ be the one-sided symbolic space that consists of all the one-sided infinite sequences
$i_{\bcdot}\colon\mathbb{N}\rightarrow\{1,\dotsc,K\}$.
From the classical Tychnoff product theorem, $\varSigma_{\!K}^+$ is a compact topological space.
For any word $({i}_1^*,\dotsc,{i}_n^*)\in\{1,\dotsc,K\}^n$ of finite-length $n\ge1$, it is well known that the corresponding cylinder set
\bean
[{i}_1^*,\dotsc,{i}_n^*]=\left\{i_{\bcdot}\in\varSigma_{\!K}^+\,|\, i_1={i}_1^*,\dotsc, i_n={i}_n^*\right\}
\eean
is an open and closed subset of $\varSigma_{\!K}^+$. One can easily check that the set of all the cylinders forms a base for the product topology of $\varSigma_{\!K}^+$.
Then, the classical one-sided Markov shift transformation given below
\bean
\theta\colon\varSigma_{\!K}^+\rightarrow\varSigma_{\!K}^+;\quad i_{\bcdot}\mapsto i_{\bcdot+1}
\eean
is continuous. By $\mathscr{B}$ we denote the Borel $\sigma$-field of the compact symbolic space $\varSigma_{\!K}^+$. A probability measure $\mu$ on $(\varSigma_{\!K}^+,\mathscr{B})$ is called \emph{$\theta$-invariant}, if $\mu(\theta^{-1}(B))=\mu(B)$ for all $B\in\mathscr{B}$. A $\theta$-invariant measure $\mu$ is called \emph{ergodic} if the only members $B$ of $\mathscr{B}$ with $\mu(B\triangle\theta^{-1}(B))=0$ satisfy $\mu(B)=0$ or $\mu(B)=1$. See \cite{NS, W82}.

For any probability measure $\mu$ on $(\varSigma_{\!K}^+,\mathscr{B})$, a point $i_{\bcdot}^*\in\varSigma_{\!K}^+$ is called a \emph{density point} of $\mu$, if for every open neighborhood $U$ around $i_{\bcdot}^*$, it holds that $\mu(U)>0$. Then, $i_{\bcdot}^*=(i_n^*)_{n=1}^{+\infty}$ is a density point of $\mu$ if and only if $\mu([i_1^*,\dotsc,i_n^*])>0$ for all $n\ge1$. Let us denote by $\mathrm{supp}(\mu)$ the set of all the density points of $\mu$ in $\varSigma_{\!K}^+$.

The following lemma  is important for our later discussion.

\begin{lem}[\cite{NS,W82}]\label{lem2.1}%%%%
Let $\mu$ be a $\theta$-ergodic probability on $(\varSigma_{\!K}^+,\mathscr{B})$. Then, $\mathrm{supp}(\mu)$, called the ``support of $\mu$", is the minimal $\theta$-invariant closed set of $\mu$-measure $1$.
\end{lem}

An infinite sequence $i_{\bcdot}=(i_n)_{n=1}^{+\infty}\in\varSigma_{\!K}^+$ is said to be \emph{periodic of period} $\pi\ge1$, if $i_{\bcdot+\pi}=i_{\bcdot}$, namely, $i_{n+\pi}=i_n$ for all $n\in\mathbb{N}$. Let $\delta_{i_{\bcdot}}$ be the Dirac probability measure concentrated at the point $i_{\bcdot}\in\varSigma_{\!K}^+$. Then, for a periodic point $i_{\bcdot}$ of period $\pi$,
\bean
\mathbb{P}_{i_{\bcdot}}:=\pi^{-1}\left(\delta_{i_{\bcdot}}+\cdots+\delta_{i_{\bcdot+\pi-1}}\right)\quad\left(~=\pi^{-1}\left(\delta_{i_{\bcdot+1}}+\cdots+\delta_{i_{\bcdot+\pi}}\right)\right)
\eean
is the unique $\theta$-ergodic probability measure whose support can be readily seen by
\bean
\mathrm{supp}(\mathbb{P}_{i_{\bcdot}})=\{i_{\bcdot},\dotsc,i_{\bcdot+\pi-1}\}\quad\left(~=\{i_{\bcdot+1},\dotsc,i_{\bcdot+\pi}\}\right).
\eean
Thus, in the periodic case, the support is quite simple. However, for a general $\theta$-ergodic probability measure $\mu$, the topological structure of its support may become very complicated and  the $\mu$-almost sure stability of linear switched dynamical system induced by $\bS$ relies on this structure, for example, see \cite{DHX08, DHX:aut11}. Our results obtained in this note further show that it is closely related to the finiteness property of a finite set of matrices.

Given two sequences $\sigma_{\!\bcdot}=(i_n)_{n=1}^{+\infty}$ and $\xi_{\bcdot}=(i_n^\prime)_{n=1}^{+\infty}$ belonging to $\varSigma_{\!K}^+$, $\xi_{\bcdot}$ is called an \emph{$\omega$-limit point} of $\sigma_{\!\bcdot}$ under $\theta$, provided that there is an infinite increasing subsequence $\{n_\ell\}_{\ell=1}^{+\infty}$ of $\mathbb{N}$ satisfying
\bean
\xi_{\bcdot}=\lim_{\ell\to+\infty}\theta^{n_\ell}(\sigma_{\!\bcdot})\quad\left(~=\lim_{\ell\to+\infty}\sigma_{\!{\bcdot+n_\ell}}\right).
\eean
Below is a lemma needed in the sequel.

\begin{lem}\label{lem2.2}%%%%%
Let $\sigma_{\!\bcdot}\in\mathrm{supp}(\mu)$ for a $\theta$-ergodic probability measure $\mu$ on $\varSigma_{\!K}^+$, and let $\xi_{\bcdot}\in\varSigma_{\!K}^+$ be a periodic sequence. If $\xi_{\bcdot}$ is an $\omega$-limit point of $\sigma_{\!\bcdot}$ under $\theta$, then $\xi_{\bcdot}$ belongs to $\mathrm{supp}(\mu)$.
\end{lem}

\begin{proof}
Let us consider the closure of the orbit of $\theta$ passing $\sigma_{\!\bcdot}$
\bean
\mathcal{O}=\mathrm{cls}\{\sigma_{\!\bcdot+n}\,|\,n=0,1,2,\dotsc\}.
\eean
Clearly, $\xi_{\bcdot}$ belongs to $\mathcal{O}$. Since by Lemma~\ref{lem2.1} $\mathrm{supp}(\mu)$ is $\theta$-invariant, we can get $\mathcal{O}\subseteq\mathrm{supp}(\mu)$. This proves Lemma~\ref{lem2.2}.
\end{proof}
%%%%%%%%%%%%%%%%%%%%%%%%%%%%%%%%%%%%%%%%%%%%%%%%%%%
\subsection{Subadditive ergodic theorem}%%%
If $f\colon X\rightarrow\mathbb{R}\cup\{-\infty\}$ is a function, we simply put $f^+(x)=\max\{0,f(x)\}$ for any $x\in X$. We need to employ another classical ergodic theorem stated as follows.

\begin{Kingman}[\cite{King,W82}]
Let $T$ be a measure-preserving transformation of a probability space $(X,\mathscr{B},\mu)$ into itself.
Let $\{f_n\}_{n\ge1}$ be a sequence of measurable functions
$$f_n\colon X\rightarrow\mathbb{R}\cup\{-\infty\}$$
satisfying the conditions:
\begin{enumerate}
\item[$\mathrm{(1)}$] $f_1^+\in L^1(\mu)$;

\item[$\mathrm{(2)}$] for each $\ell,m\ge1$,
$f_{\ell+m}(x)\le f_\ell(x)+f_m(T^\ell(x))$
for $\mu$-a.e. $x\in X$.
\end{enumerate}
Then, there exists a measurable function
\bean
\bar{f}\colon X\rightarrow\mathbb{R}\cup\{-\infty\}
\eean
such that
\begin{align*}
&{\bar{f}}^+\in L^1(\mu),\quad \bar{f}\circ T=\bar{f}\textrm{ a.e.},\quad \lim_{n\to+\infty}\frac{1}{n}f_n=\bar{f}\textrm{ a.e.},\\
\intertext{and}
&\lim_{n\to+\infty}\frac{1}{n}\int_Xf_n\,d\mu=\inf_{n\ge1}\frac{1}{n}\int_Xf_n\,d\mu=\int_X\bar{f}\,d\mu.
\end{align*}
\end{Kingman}

\begin{notes}
If $\mu$ is $T$-ergodic, then $\bar{f}(x)\equiv\mathrm{constant}$ for $\mu$-a.e. $x\in X$.
\end{notes}
\begin{notes}
This result in fact implies the classical Multiplicative Ergodic Theorem.
\end{notes}
%%%%%%%%%%%%%%%%%%%%%%%%%%%%%%%%%%%%%%%%%%%%%%%%%%%%
\subsection{A realization of the joint spectral radius via ergodic measures}%%%

The foundation of our argument later is the following realization theorem of the joint (or generalized) spectral radius:

\begin{theorem}[\cite{DHX:ERA11}]\label{thm2.3}%%%
Let $\bS=\{S_1,\ldots,S_K\}\subset\mathbb{C}^{d\times d}$ be arbitrary. Then, there
is at least one ergodic Borel probability measure $\mu_*$ of the one-sided Markov shift
$\theta\colon\varSigma_{\!K}^+\rightarrow\varSigma_{\!K}^+$, which is extremal for $\bS$.
\end{theorem}

The extremal measure is useful for the study of dynamical behaviors of $\bS$, such as the rotation number of extremal trajectories and extremal switching sequences of $\bS$, as is shown in \cite{Koz07, Dai:jde11}.
%%%%%%%%%%%%%%%%%%%%%%%%%%%%%%%%%%%%%%%%%%%%%%%%%%%%%%%%
%%%%%%%%%%%%%%%%%%%%%%%%%%%%%%%%%%%%%%%%%%%%%%%%%%%%%%%%
\section{Extremal norms}\label{sec3}%%%%
Let $\bS=\{S_1,\dotsc,S_K\}\subset\mathbb{C}^{d\times d}$ where $d\ge2$ is an integer. As a special case of the main result of Rota and Strang~\cite{RS60}, there holds
\begin{equation*}
\hat{\rho}(\bS)=\inf_{\norm{\cdot}\in\mathcal{N}}\max_{1\le k\le K}\norm{S_k}
\end{equation*}
where $\mathcal{N}$ denotes the set of all possible matrix norms for $\mathbb{C}^{d\times d}$ induced by vector norms on $\mathbb{C}^d$; also see \cite{El, SWP} for a short proof. Hence, an important problem is whether or not the above infimum is actually attained by some induced matrix norm $\norm{\cdot}$. For this, a norm $\pmb{\|}\cdot\pmb{\|}$ on $\mathbb{C}^d$ satisfying the condition
\begin{equation*}
\hat{\rho}(\bS)=\max_{1\le k\le K}\pmb{\|}S_i\pmb{\|}\quad\left(~=\max_{i_{\bcdot}\in\varSigma_{\!K}^+}\sqrt[n]{\pmb{\|}S_{i_1}\cdots S_{i_n}\pmb{\|}}\; \forall n\ge1\right),
\end{equation*}
is called an \emph{extremal norm} of $\bS$. From Barabanov's extremal norm theorem~\cite{Bar}, one can see that if $\bS$ is ``irreducible", i.e., there are no nontrivial, common, and proper subspaces of $\mathbb{C}^{d}$ for each member $S_k$ of $\bS$, then there exists an extremal norm for $\bS$ on $\mathbb{C}^{d}$. Here the irreducibility is crucial for Barabanov's theorem as shown by a simple counterexample example
\bean
\bS=\left\{\left[\begin{matrix}1&1\\ 0&1\end{matrix}\right]\right\},
\eean
which does not have any extremal norms. However, there always exists a lower-dimensional extremal norm.

\begin{theorem}\label{thm3.1}%%%
For $\bS=\{S_1,\dotsc,S_K\}\subset\mathbb{C}^{d\times d}$, there exists an invariant linear subspace $\mathbb{E}$ of $\mathbb{C}^d$ on which there is a norm $\pmb{|}\cdot\pmb{|}$ such that
\bean
\hat{\rho}(\bS)=\hat{\rho}(\bS\!\upharpoonright\!\mathbb{E})=\max_{1\le k\le K}\pmb{|}{{S_k}}\!\upharpoonright\!\mathbb{E}\pmb{|}.
\eean
\end{theorem}

\begin{proof}
If $\rho(\bS)=0$ then we simply let $\mathbb{E}=\{\mathbf{0}\}$ where $\mathbf{0}$ is the origin of $\mathbb{C}^d$. So, it may be assumed, without loss of generality, that $\rho(\bS)>0$.
We denote by $\widehat{\bS}$ the set $\{S_1/\rho(\bS),\dotsc, S_K/\rho(\bS)\}\subset\mathbb{C}^{d\times d}$. Clearly, $\rho(\widehat{\bS})=1$.

If the multiplicative semigroup ${\widehat{\bS}}^+$, generated by $\widehat{\bS}$, is bounded in $\mathbb{C}^{d\times d}$, then there exists an extremal norm for $\bS$ on $\mathbb{C}^d$ (see, e.g., \cite{Wirth02}) and we define $\mathbb{E}=\mathbb{C}^d$ in this case. Otherwise, suppose that ${\widehat{\bS}}^+$ is unbounded. Applying Elsner's reduction theorem~\cite{El} repeatedly, there exists a nonsingular matrix $P\in\mathbb{C}^{d\times d}$ and $2\le r\le d$ such that
\bean
P^{-1}{\widehat{S}}_kP=\left[\begin{matrix}\widetilde S_k^{(1)}&0&\cdots&0\\ \spadesuit_k&\widetilde S_k^{(2)}&\cdots&0\\\vdots&\vdots&\ddots&\vdots\\ * &*&\cdots&\widetilde S_k^{(r)}\end{matrix} \right]
\eean
where $\widetilde{S}_k^{(j)}$ is $d_j\times d_j$ for $1\le j\le r$, $d_1+\cdots+d_r=d$ and $\widetilde{\bS}(j):=\left\{\widetilde{S}_1^{(j)},\dotsc,\widetilde{S}_K^{(j)}\right\}$
generates a bounded semigroup ${\widetilde{\bS}(j)}^+$ in $\mathbb{C}^{d_j\times d_j}$. Let us notice here that
\begin{equation*}\rho(\widehat{\bS})=1=\max_{1\le j\le r} \rho(\widetilde{\bS}(j)).\end{equation*}
Thus there exists at least one $j$ such that
$\rho(\widetilde{\bS}(j))=1$. Now, from \cite[Lemma~6.2]{Wirth02} one can obtain the desired result. To be self-contained, we provide a detailed proof here.

Without loss of generality, we may assume that $\rho(\widetilde{\bS}(1))<1$ and $\rho(\widetilde{\bS}(2))=1$; this is because the other cases can always be reduced to this one. We now write
\bean
\bF=\left\{F_k=\left[\begin{matrix}\widetilde S_k^{(1)}&0\\ \spadesuit_k&\widetilde S_k^{(2)}\end{matrix} \right]\colon 1\le k\le K\right\}.
\eean
Then, one can easily see that $\rho(\bF)=1$ and the semigroup $\bF^+$ generated by $\bF$ is bounded in $\mathbb{C}^{(d_1+d_2)\times(d_1+d_2)}$. Then it follows that there is an extremal norm $\pmb{\|}\cdot\pmb{\|}$ for $\bF$ on $\mathbb{C}^{(d_1+d_2)\times(d_1+d_2)}$.
This implies that
\bean
\hat{\rho}(\bS)=\sup_{1\le k\le K}\pmb{\|}(P^{-1}S_kP)\!\upharpoonright\!\mathbb{C}^{d_1+d_2}\pmb{\|}.
\eean
Finally, let $\mathbb{E}=P^{-1}(\mathbb{C}^{d_1+d_2}\times\{\mathbf{0}\})$ where $\mathbf{0}$ is the origin of $\mathbb{C}^{d-d_1-d_2}$ and define
\begin{equation*}
\pmb{|}x\pmb{|}=\pmb{\|}P(x)\pmb{\|}\quad\left(~=\pmb{\|}xP\pmb{\|}\right)\quad \forall x\in\mathbb{E},
\end{equation*}
where $x\in\mathbb{C}^d$ is viewed as a row vector. Clearly, such $\mathbb{E}$ and $\pmb{|}\cdot\pmb{|}$ satisfy the requirement.

This completes the proof of Theorem~\ref{thm3.1}.
\end{proof}

Another generalization of Barabanov's theorem can be found in \cite{Dai-JMAA}. In addition, to prove our main theorem, we will need a reduction theorem of ergodic version.

\begin{lem}[{\cite[Lemma~3.5]{DHX:aut11}}]\label{lem3.2}%%%%
Let $\bS=\{S_1,\dotsc,S_K\}\subset\mathbb{C}^{d\times d}$ and $\mu_*$ an extremal ergodic probability for $\bS$. Then, there exists a nonsingular matrix $P\in\mathbb{C}^{d\times d}$ and $2\le r\le d$ such that
\bean
P^{-1}S_kP=\left[\begin{matrix}\widetilde S_k^{(1)}&0&\cdots&0\\ *&\widetilde S_k^{(2)}&\cdots&0\\\vdots&\vdots&\ddots&\vdots\\ * &*&\cdots&\widetilde S_k^{(r)}\end{matrix} \right],\quad 1\le k\le K
\eean
where
\begin{equation*}
\widetilde{\bS}(j)=\left\{\widetilde{S}_1^{(j)},\dotsc,\widetilde{S}_K^{(j)}\right\}\subset\mathbb{C}^{d_j\times d_j}
\end{equation*}
is irreducible for every $1\le j\le r$, $d_1+\cdots+d_r=d$, and that there exists some $1\le j\le r$ satisfying that $\rho(\bS)=\rho(\widetilde{\bS}(j))$ and $\mu_*$ is also extremal for $\widetilde{\bS}(j)$.
\end{lem}
%%%%%%%%%%%%%%%%%%%%%%%%%%%%%%%%%%%%%%%%%%%%%%%%%%%%%%%
%%%%%%%%%%%%%%%%%%%%%%%%%%%%%%%%%%%%%%%%%%%%%%%%%%%%%%%
\section{Spectral finiteness property}\label{sec4}%

In this section, we will present a sufficient and necessary condition for the spectral finiteness property of a finite set $\bS$ of $d$-by-$d$ matrices.
The main result of this note can be stated as follows:

\begin{Mainthm}%%%
Let $\bS=\{S_1,\ldots,S_K\}\subset\mathbb{C}^{d\times d}$ be an arbitrary finite family of matrices. Then, the following two statements are equivalent to each other.
\begin{enumerate}
\item[$\mathrm{(1)}$] $\bS$ has the spectral finiteness property, i.e., there is a word $({i}_1^*,\dotsc,{i}_n^*)$ in $\{1,\dotsc,K\}^n$ for some $n\ge1$ such that
\bean
\rho(\bS)=\sqrt[n]{\rho(S_{{i}_1^*}\dotsm S_{{i}_n^*})}.
\eean

\item[$\mathrm{(2)}$] There is an extremal $\theta$-ergodic probability measure $\mu_*$ for $\bS$ on $\varSigma_{\!K}^+$, which has a periodic density point $\xi_{\bcdot}\in\varSigma_{\!K}^+$.
\end{enumerate}
\end{Mainthm}

According to the realization theorem (Theorem~\ref{thm2.3}), there always exists at least one extremal $\theta$-ergodic probability measure $\mu_*$ for $\bS$, on the symbolic space $\varSigma_{\!K}^+$.

\begin{proof}
If $\rho(\bS)=0$ then $\rho(S_k)=0$ for all $1\le k\le K$. So in this case, the statement of Main theorem holds trivially. Thus we next assume $\rho(\bS)>0$.

$\mathrm{(1)}\Rightarrow\mathrm{(2)}$. Let there exist a word $(i_1^\prime,\ldots,i_\pi^\prime)\in\{1,\dotsc,K\}^\pi$ for some $\pi\ge1$ satisfying
\bean
\rho(\bS)=\sqrt[\pi]{\rho(S_{i_1^\prime}\cdots S_{i_\pi^\prime})}.
\eean
Then we consider the periodic switching sequence
\bean
\xi_{\bcdot}=(i_n)_{n=1}^{+\infty}\in\varSigma_{\!K}^+
\quad\textrm{with }
i_{n+\ell\pi}=i_n^\prime\ \forall \ell\ge0\textrm{ and }1\le n\le\pi.
\eean
For any $\ell\ge0$, we have $\xi_{\bcdot+\ell}=(i_{n+\ell})_{n=1}^{+\infty}$ and further
from the classical Gel'fand spectral radius formula it follows that
\bean
\lim_{n\to+\infty}\sqrt[n]{\|S_{i_{1+\ell}}\dotsm S_{i_{n+\ell}}\|_2}=\rho(\bS),
\eean
where $\|\cdot\|_2$ denotes the matrix norm induced by the standard Euclidean vector norm on $\mathbb{C}^d$. Thus,
\bean
\mathbb{P}_{\xi_{\bcdot}}=\pi^{-1}\left(\delta_{\xi_{\bcdot}}+\cdots+\delta_{\xi_{\bcdot+\pi-1}}\right)
\eean
is an extremal $\theta$-ergodic probability measure for $\bS$.

$\mathrm{(2)}\Rightarrow\mathrm{(1)}$. First of all, there is no loss of generality in assuming $\rho(\bS)=1$ and suppose that $\bS$ has an extremal $\theta$-ergodic probability measure $\mu_*$ on $\varSigma_{\!K}^+$, which has a periodic density point $i_{\bcdot}^\prime=(i_n^\prime)_{n=1}^{+\infty}$ with $i_{n+\pi}^\prime=i_n^\prime$ for $n\ge1$ and some $\pi\ge1$. By Lemma~\ref{lem3.2} it might be assumed, without loss of generality, that $\bS$ is irreducible and then it has an extremal norm $\pmb{\|}\cdot\pmb{\|}$ on $\mathbb{C}^d$ such that
\bean
1=\hat{\rho}(\bS)=\max_{1\le k\le K}\pmb{\|}S_k\pmb{\|}.
\eean
From the multiplicative ergodic theorem~\cite{FK} and the Kingman subadditive ergodic theorem, it follows immediately that
\bean
\log\rho(\bS)=\inf_{n\ge1}\frac{1}{n}\int_{\varSigma_{\!K}^+}\log\pmb{\|}S_{i_1}\dotsm S_{i_n}\pmb{\|}\,d\mu_*(i_{\bcdot}).
\eean
Therefore,
\bean
1=\inf_{n\ge1}\prod_{1\le i_1,\dotsc,i_n\le K}\pmb{\|}S_{i_1}\cdots S_{i_n}\pmb{\|}^{\mu_*([i_1,\ldots,i_n])/n}
\eean
where $[i_1,\dotsc,i_n]$ denotes the cylinder set defined by the word $(i_1,\dotsm,i_n)$ as in Section~\ref{sec2.1}. Hence, we have
\bean
1\le\prod_{1\le i_1,\dotsc,i_n\le K}\pmb{\|}S_{i_1}\dotsm S_{i_n}\pmb{\|}^{\mu_*([i_1,\dotsm,i_n])/n}\quad\forall n\ge1.
\eean
Since $i_{\bcdot}^\prime=(i_n^\prime)_{n=1}^{+\infty}$ is a density point of $\mu_*$ with $i_{n+\pi}^\prime=i_n^\prime$ for all $n\ge1$, we can obtain that
\bean
1=\sqrt[n\pi]{\pmb{\|}(S_{i_1^\prime}\dotsm S_{i_{\pi}^\prime})^n\pmb{\|}}\quad\forall n\ge1.
\eean
This implies from the Gel'fand formula that
\bean
1=\sqrt[\pi]{\rho(S_{i_1^\prime}\cdots S_{i_{\pi}^\prime})}\le\rho(\bS)
\eean
which shows that $\bS$ has the spectral finiteness property.

The proof of Main theorem is thus completed.
\end{proof}

For any irreducible transition probability matrix $\mathrm{P}=(p_{ij})\in\mathbb{R}^{K\times K}$ with a stationary distribution $\mathbf{p}=(p_1,\ldots,p_K)$ (in which all $p_k>0$), one can define a $\theta$-ergodic probability measure $\mu_{\mathbf{p},\mathrm{P}}$ on $\varSigma_{\!K}^+$ by the following way:
\bean
\mu_{\mathbf{p},\mathrm{P}}([i_1])=p_{i_1}\ \textrm{for }n=1\quad \textrm{and}\quad\mu_{\mathbf{p},\mathrm{P}}([i_1,\dotsc,i_n])=p_{i_1}p_{i_1i_2}\dotsc p_{i_{n-1}i_n}\ \textrm{for }n\ge2,
\eean
for all words $(i_1,\dotsc,i_n)\in\{1,\dotsc,K\}^n$.
Such $\mu_{\mathbf{p},\mathrm{P}}$ is called a canonical $(\mathbf{p},\mathrm{P})$-Markovian probability, which is $\theta$-ergodic (cf.~\cite[Theorem~1.13]{W82}).

The results of our Main theorem lead to the following two corollaries.

\begin{cor}\label{cor4.2}
Let $\bS=\{S_1,\ldots,S_K\}\subset\mathbb{C}^{d\times d}$. If $\mu_{\mathbf{p},\mathrm{P}}$ is extremal for $\bS$, then $\bS$ has the spectral finiteness property.
\end{cor}

\begin{proof}
According to the classical theory of symbolic dynamics, $\mathrm{supp}(\mu_{\mathbf{p},\mathrm{P}})=\varSigma_{\!K}^+$. Combining with Main theorem, this completes the proof of Corollary~\ref{cor4.2}.
\end{proof}

\begin{cor}\label{cor4.3}
Let $\bS=\{S_1,\ldots,S_K\}\subset\mathbb{C}^{d\times d}$ satisfy the periodically switched stability condition:
\bean
\rho(S_{i_1}\cdots S_{i_n})<1\quad \forall (i_1,\ldots,i_n)\in\{1,\dotsc,K\}^n\textrm{ and }n\ge1.
\eean
If $\mu_{\mathbf{p},\mathrm{P}}$ is extremal for $\bS$, then $\rho(\bS)<1$.
\end{cor}

\begin{proof}
According to Corollary~\ref{cor4.2}, it follows that $\bS$ has the spectral finiteness property, and thus $\rho(\bS)<1$ by the periodically switched stability. Therefore, the conclusion follows.
\end{proof}
%%%%%%%%%%%%%%%%%%%%%%%%%%%%%%%%%%%%%%%%%%%%%%%%%%%%%%%%
%%%%%%%%%%%%%%%%%%%%%%%%%%%%%%%%%%%%%%%%%%%%%%%%%%%%%%%%
\section{Conclusion remarks}\label{sec5}%

In this note, we have proved that $\bS=\{S_1,\dotsc,S_K\}\subset\mathbb{C}^{d\times d}$ has the spectral finiteness property if and only if there exists a periodic switching sequence in the support of some extremal ergodic measure $\mu_*$ of $\bS$. This result reveals that the topological structure of an extremal (ergodic) measures is closely related to the finiteness property, and the dynamical stability of linear switched dynamical system induced by $\bS$ is characterized by this topological structure.

An explicit important example can further illustrate this. Let us consider
\bean
\bS(\alpha)=\left\{\alpha\left[\begin{matrix}1&1\\ 0&1\end{matrix}\right],\ \left[\begin{matrix}1&0\\ 1&1\end{matrix}\right]\right\}\quad \textrm{where }\alpha\in\mathbb{R}.
\eean
From \cite{BM, BTV, Koz07, HMST}, it follows that there is an $\balpha>0$ such that
$\bS(\balpha)$ does not have the spectral finiteness property. Thus, according to the main theorem given in this note, each extremal $\theta$-ergodic probability of $\bS(\balpha)$ does not have any periodic density points.

For the symbolic sequence space $\varSigma_{\!K}^+$, the observable measures are the canonical Markovian probabilities $\mu_{\mathbf{p},\mathrm{P}}$ as defined in Section~\ref{sec4}, and such measures all have periodic density points. Therefore, it is necessary to study the topological structure of those non-observable measures on $\varSigma_{\!K}^+$. This indicates that one has to look at these  non-observable measures on $\varSigma_{\!K}^+$ in order to disprove the finiteness property of $\bS$. In current literature the disapproval of finiteness property essentially makes use of this fact implicitly \cite{BM, BTV, Koz07, HMST}.
%%%%%%%%%%%%%%%%%%%%%%%%%%%%%%%%%%%%%%%%%%%%%%%%%%%%%%%%%%%%
\baselineskip=0.9\normalbaselineskip %produces single spacing

%%%%%%%%%%%%%%%%%%%%%%%%%%%%%%%%%%%%%%%%%%%%%%%%%%%%%%%
%%%%%%%%%%%%%%%%%%%%%%%%%%%%%%%%%%%%%%%%%%%%%%%%%%%%%%%

\end{document}